\documentclass{amsart}
\usepackage[utf8]{inputenc}

\title{L\'evy Flows and associated Stochastic PDEs}
\author{Arvind Kumar Nath$^\ast$}
\address{Arvind Kumar Nath \footnote{Corresponding author}, Department of Mathematics and Statistics, Indian Institute of Technology Kanpur, Kalyanpur, Kanpur - 208016, India.}
\email{yarvind@iitk.ac.in}
\thanks{$^\ast$First author is the Corresponding Author}

\author{Suprio Bhar}
\address{Suprio Bhar, Department of Mathematics and Statistics, Indian Institute of Technology Kanpur, Kalyanpur, Kanpur - 208016, India.}
\email{suprio@iitk.ac.in}

\date{}
\usepackage{setspace}
\usepackage{geometry}
\usepackage{enumerate}
\usepackage{enumitem, xcolor, amssymb,latexsym,amsmath,bbm}
\usepackage{mathtools}
\usepackage[mathscr]{euscript}
\usepackage{amsthm}
\usepackage{amssymb}
\usepackage{enumitem}
\usepackage{cite}
\newcommand{\R}{\mathbb{R}}

\newcommand{\pp}{\mathbb{P}}
\newcommand{\Exp}{\mathbb{E}}
\newcommand{\F}{\mathcal{F}}
\newcommand{\Sc}{\mathcal{S}}

\newcommand{\inpr}[3][]{\left\langle#2 \,,\, #3\right\rangle_{#1}}

\newtheorem{theorem}{Theorem}[section]

\newtheorem{lemma}[theorem]{Lemma}
\newtheorem{proposition}[theorem]{Proposition}
\theoremstyle{definition}
\newtheorem{definition}[theorem]{Definition}
\theoremstyle{remark}
\newtheorem{remark}[theorem]{Remark}

\numberwithin{equation}{section}

\allowdisplaybreaks[4]

\begin{document}

\begin{abstract} 
In this paper, we first explore certain structural properties of L\'evy flows and use this information to obtain the existence of strong solutions to a class of Stochastic PDEs in the space of tempered distributions, driven by L\'evy noise. The uniqueness of the solutions follows from Monotonicity inequality. These results extend an earlier work of the second author on the diffusion case.
\end{abstract}

\maketitle

\section{Introduction}

Let $(\Omega, \F, (\F)_t, \pp)$ be a complete filtered probability space satisfying the usual conditions. Let $\{B_t\}_t$ be a standard $d$-dimensional $(\F_t)_t$ Brownian motion. Let $N$ be a Poisson random measure with on $\R^d$ with associated L\'evy measure $\nu$ on $\R^d \setminus \{0\}$ and let $\widetilde N$ denote the corresponding compensated Poisson random measure. We assume that $B$ and $N$ are independent. For any $z \in \R^d$, consider the following stochastic differential equation (SDE) in $\R^d$
\begin{equation}\label{general-SDE}
\begin{split}
dZ_t &= b( Z_{t-} ) dt + \sigma( Z_{t-} ) dB_t + \int_{ 0<|x| < 1 } F(Z_{ t - } , x ) \widetilde{ N }( dt,dx) ,\\
Z_0 &= z,
\end{split}
\end{equation}
where the coefficients $b:\R^d \to \R^d, \sigma:\R^d \to \R^{d \times d}, F:\R^d \times \{x \in \R^d: 0 < |x| < 1\} \to \R^d$ satisfies the Lipschitz condition and linear growth condition as described below, i.e. there exist constants $K_1 > 0$ and $K_2 > 0$ such that for all $z, z_1, z_2 \in \R^d$ we have
\begin{equation}\label{lipschitz-condition}
    \left|b(z_1)-b(z_2)\right|^2 + \|a(z_1,z_1)-2a(z_1,z_2)+a(z_2,z_2)\|^2 + \int_{0<|x|<1}\left|F(z_1,x)- F(z_2,x)\right|^2\nu(dx)\leq K_1\left|z_1-z_2\right|^2 
 \end{equation}
and
\begin{equation}\label{lin-growth-condition}
    \left|b(z)\right|^2 + \|a(z,z)\|^2 + \int_{0<|x|<1}\left|F(z,x)\right|^2\nu(dx)\leq K_2(1+|z|^2).
\end{equation}
Here, $a(x,y) := \sigma(x)\sigma^t(y), \forall x, y \in \R^d$, $\sigma^t$ being the transpose of $\sigma$. The term $|x|$ denotes the Euclidean norm of a vector $x$ and $\|A\|$ denotes the matrix seminorm $\sum_{i = 1}^d |a_{ii}|$ of any real square matrix $A = (a_{ij})$ of order $d$. The component functions of $b, \sigma$ and $F$ shall be denoted by $b_i, \sigma_{ij}, F_i; i, j = 1, \cdots, d$, respectively.

We denote the unique strong solution \cite[Theorem 6.2.3]{MR2512800} to the SDE \eqref{general-SDE} above by $\{Z^z_t\}_t$. In the first half of the present paper, we look at the flows associated to this SDE. In stochastic analysis/dynamics, stochastic flows play an important role (see \cite{MR2531088, MR866340, MR608026, MR805125, MR517235, MR776981, MR876080, MR1472487, MR2020294, MR3929750, MR540035, MR622556} and the references therein). In this article, we are interested in specific structures on L\'evy flows (see \cite{MR1276439, MR1251218, MR1093330, MR1045162, MR2090755, MR3929750, MR2460554, MR2512800} and the references therein) and we extend results of \cite{MR3647575} on the diffusion case.

In \cite{MR3647575}, for diffusions, i.e. for $F \equiv 0, G \equiv 0$, a decomposition of $\{Z^z_t\}_t$ into a `global' component $\{Z^0_t\}_t$ and a deterministic `local' component $\{Z^z_t - Z^0_t\}_t$ was considered. This situation was characterized by \cite[Theorem 2.5]{MR3647575}, which implies that the diffusion coefficient $\sigma$ is a real non-singular matrix of order $d$ and the drift coefficient $b$ is in an affine form, viz. $b(x) = \beta + Ax, \forall x \in \R^d$ for some $\beta \in \R^d$ and a real square matrix $A$ of order $d$. Since the flows generated in this class of SDEs are Gaussian, this result can also be treated as a characterization result for Gaussian flows. The well-known Ornstein-Uhlenbeck diffusion \cite[p. 241]{MR2512800} falls in this class.

In this paper, we extend the above characterization to the case of L\'evy flows. We start with an analogue of \cite[Definition 2.2]{MR3647575}.

\begin{definition}\label{gen-ddi}
We say that the  solution $\{Z^z_t\}_t$ to the SDE (\ref{general-SDE}) depends deterministically on the initial condition if there exists a function $f:[0,\infty)\times\R^d \to \R^d$ such that for all $z\in \R^d$, we have a.s.
\begin{equation}\label{dependence deterministically on initial condition dd}
    Z_t^z ( \omega ) = f(t,z) + Z^0_t( \omega )\; \forall t \geq 0.
\end{equation}
We take $f$ to be jointly measurable in $t$ and $z$.
\end{definition}

Based on the assumptions on the coefficients, if the deterministic dependence on the initial condition holds as above, then necessarily for all $z \in \R^d$ and $t \geq 0$, we have 
\begin{equation}
\begin{split}
f(t,z) &= \Exp(Z^z_t) - \Exp(Z^0_t)\\
&= z + \int^t_0 \left( \Exp(b(Z^z_{s-})) - \Exp(b(Z^0_{s-})) \right) ds
\end{split}
\end{equation}
For every fixed $z$, $t \mapsto \Exp(b(Z^z_{t-}))-\Exp(b(Z^0_{t-}))$ is left continuous with right limits and hence, the left partial derivative of $f$ with respect to $t$ exists on $(0, \infty)$. We shall denote this derivative by $\frac{\partial f}{\partial t}$ and take $\frac{\partial f}{\partial t}(0, z) = 0, \forall z$. Note that, $\frac{\partial f}{\partial t}(t, z) = \Exp(b(Z^z_{t-}))-\Exp(b(Z^0_{t-})), t > 0$. 

We now state the main result of the first half of this paper.

\begin{theorem}\label{ddi-characterization}
Let $\sigma, b,F$ be as stated above. Suppose the following happen:
\begin{enumerate}[label=(\roman*)]
    \item there exists $z\in\R^d$ such that the determinant of  $\sigma(z)=(\sigma_{ij}(z))$ is not zero,
    \item $b_i\in C^2(\R^d,\R),i=1,2\cdots,d$,
    \item for every fixed $z\in\R^d$, the map $t \mapsto \frac{\partial f}{\partial t}(t,z) $ is of bounded variation.
\end{enumerate}
Then the solution of the SDE \eqref{general-SDE} depends deterministically on the initial condition through \eqref{dependence deterministically on initial condition dd} if and only if $\sigma$ is a real non-singular matrix of order $d$, $b$ is of the form $b(z)=Cz+K$,  $F(z,x)$ depends only on $x$ and  $f(t,z)=e^{tC}z \; \forall t\geq 0,  z\in \R^d$  where $K \in\R^d$ and $C$ is a real matrix of order $d$.
\end{theorem}

The proof of this result is similar to \cite[Theorem 2.5]{MR3647575} and has been discussed in Section 2. To avoid repetitions of similar arguments from \cite{MR3647575}, a sketch of the proof is provided for the one-dimensional case only.

Next, we discuss the existence and uniqueness results of strong solutions for a class of stochastic partial differential equations (SPDEs) (see \cite{MR570795, MR1999259, MR1135324, MR2479730, MR3236753, MR2373102, MR2356959, MR1465436, MR2560625, MR2857016, MR2765423, MR2674056} and the references therein), which are an extension of the results considered in \cite[Section 3]{MR3647575}. To do this, we first recall a countably Hilbertian topology on the space $\Sc(\R^d)$ of real valued rapidly decreasing smooth 
functions on $\R^d$ and its dual $\Sc^\prime(\R^d)$, the space of tempered distributions.

For any $p \in \R$, let the Hermite-Sobolev spaces  $\Sc_p(\R^d)$ be the 
completion of 
$\Sc(\R^d)$ in the inner product $\inpr[p]{\cdot}{\cdot}$ which is defined in 
terms of the $L^2(\R^d)$ inner product (see
\cite[Chapter 1.3]{MR771478} for the details). The spaces $\Sc_p(\R^d), p 
\in 
\R$ are real separable Hilbert spaces. Note that for 
$p>0$, $\Sc_p(\R^d) \subset L^2(\R^d)$ and $(\Sc_{-p}(\R^d), \|\cdot\|_{-p})$ is isometrically isomorphic to the dual to $(\Sc_p(\R^d),
\|\cdot\|_p)$. This duality extends from the identification of $\Sc_0(\R^d) = L^2(\R^d)$ duality with itself. Furthermore,
\[\Sc(\R^d) = \bigcap_{p \in \R}(\Sc_p(\R^d), \|\cdot\|_p), \quad 
\Sc'(\R^d) = \bigcup_{p \in \R}(\Sc_p(\R^d), \|\cdot\|_p)
\]
Given $\psi \in \Sc(\R^d)$ (or $\Sc_p(\R^d)$) and $\phi \in \Sc'(\R^d)$
(or
$\Sc_{-p}(\R^d)$), the action of $\phi$ on $\psi$ will be denoted by
$\inpr{\phi}{\psi}$. 

Now, we consider a class of SPDEs in the space of tempered distributions $\Sc^\prime(\R^d)$ associated to the SDE \eqref{general-SDE} given above. Proofs of new results have been discussed in Section 3. Take $\sigma, b, F$ as in Theorem \ref{ddi-characterization}. Note that 
\begin{equation}\label{integrable-F}\int_{0<|x|<1}\left|F(0,x)\right|^2\nu(dx)< \infty.
\end{equation}
In addition, we assume that
\begin{equation}\label{bound-F}\sup_{0<|x|<1}\left|F(0,x)\right|< \infty.
\end{equation}
Now, look at the linear SPDE

\begin{equation}\label{general-SPDE}
 \begin{split}
  d{Y}_t &= A^*({Y}_{t-})dB_t + L^*({Y}_{t-})dt +\int_{0<|x|<1}(\tau_{F(0,x)}-Id){Y}_{t-} \widetilde{N}(dt,dx) 
\\
&\qquad +\int_{0<|x|<1}(\tau_{F(0,x)}-Id + \sum_{i = 1}^d F_i(0,x) \partial_i){Y}_{t-}\nu(dx)dt,\\
Y_0 &= \psi,
\end{split}
 \end{equation}
with the terms described as below.

\begin{enumerate}[label=(\roman*)]
    \item We take the initial condition $\psi$ to be in $L^1(\R^d)$. Note that $L^1(\R^d) \subset \Sc_{-p}(\R^d)$ for any $p > \frac{d}{4}$, see \cite[Lemma 3.2]{MR3647575}.
    \item Given $\phi \in \Sc^\prime(\R^d)$, the distributional derivatives $\partial_i \phi \in \Sc^\prime(\R^d), i = 1, \cdots, d$ are defined by the duality relation
    \[\inpr[]{\partial_i \phi}{g} = -\inpr[]{\phi}{\partial_i g}, \forall g \in \Sc(\R^d).\]
    Similarly, the multiplication operators $M_i, 
i=1,\cdots,d$ 
defined by $(M_i g)(x):= x_i g(x),\, g\in \Sc(\R^d), x=(x_1,\cdots,x_d) 
\in \R^d$ on $\Sc(\R^d)$ are extended by duality to $M_i:\Sc'(\R^d)\to\Sc'(\R^d)$, where
\[\inpr[]{M_i \phi}{g} = \inpr[]{\phi}{M_i g}, \forall g \in \Sc(\R^d).\]
It is well-known that $\partial_i,M_i:\Sc_{p}(\R^d)\to\Sc_{p-\frac{1}{2}}(\R^d)$ are bounded linear 
operators, 
for any $p \in 
\R$.
    \item Define operators $A,L$ as follows:
for $\phi \in 
\Sc(\R^d)$ and $x \in \R^d$,
\begin{equation}
\begin{cases}
A\phi := (A_1\phi,\cdots,A_d\phi),\\
A_i\phi(x) := \sum_{k=1}^d \sigma_{ki}\partial_k\phi(x),\\
L\phi(x) := \frac{1}{2}\sum_{i,j=1}^d 
(\sigma\sigma^t)_{ij}\partial^2_{ij}\phi(x)+\sum_{i=1}^d
b_i(x)\partial_i\phi(x).
\end{cases}
\end{equation}
On $\Sc^\prime(\R^d)$, define the adjoint 
operators $A^{\ast}, L^\ast$ as follows: for $\psi \in 
\Sc^\prime(\R^d)$,
\begin{equation}\label{adj-ops}
\begin{cases}
A^{\ast}\psi := (A_1^{\ast}\psi,\cdots,A_d^{\ast}\psi),\\
A_i^{\ast}\psi := -\sum_{k=1}^d \partial_k\left(\sigma_{ki}\psi\right),\\
L^{\ast}\psi := 
\frac{1}{2}\sum_{i,j=1}^d \partial^2_{ij}\left((\sigma\sigma^t)_{ij}
\psi\right)-\sum_{i=1}^d
\partial_i\left(b_i\psi\right).
\end{cases}
\end{equation}
For any $p \in \R$, $A^{\ast}_i: \Sc_{-p}(\R^d) \to \Sc_{p-\frac{1}{2}}(\R^d), i = 1, 2, \cdots, d$ and $L^\ast:\Sc_{p}(\R^d) \to \Sc_{p-1}(\R^d)$ are bounded linear operators.

\item Translation operators $\tau_x:\Sc'(\R^d)\to \Sc'(\R^d), x \in \R^d$ are defined by the duality relation
\[\inpr[]{\tau_x \phi}{g} = \inpr[]{\phi}{\tau_{-x} g}, \forall \phi \in \Sc^\prime(\R^d), g \in \Sc(\R^d),\]
where $(\tau_x g)(y) := g(y-x),\, y \in \R^d$.
\end{enumerate}

Motivated by \cite{MR2373102, MR3647575}, we construct a solution to \eqref{general-SPDE} in the following way. There exists a $\pp$ null set $\mathcal{N}$ such that for all $\omega\in \Omega\setminus\mathcal{N},t\geq 0$
\[Z^0_t=  \int^t_0\sigma dB_s + \int^t_0(CZ^0_{s-}+K) ds +\int^t_0\int_{ 0<|x| < 1 } F(0 , x ) \widetilde{ N }( ds,dx)\]
We take $\{Z^z_t\}_t$ with $Z^z_t := Z^0_t+ze^{tC}, \forall t \geq 0, z \in \R^d$ to be the solution to the SDE \eqref{general-SDE}. Now define, 
\begin{equation}\label{Y-psi}
Y_t(\omega)(\psi) := \int_{\R^d} \psi(z) \delta_{Z^z_t(\omega)}\, dz, \forall t \geq 0, \omega \in \Omega\setminus\mathcal{N}
\end{equation}
and set $Y_t(\omega)(\psi) = 0, \forall t \geq 0, \omega \in \mathcal{N}$. Here, for $z \in \R^d$, let $\delta_z$ denotes the Dirac distribution supported at the point $z$. For $p > \frac{d}{4}$, $\sup_{z \in \R^d}\|\delta_z\|_{-p} \leq C(p)$ for some positive constant $C(p)$ depending on $p$ (see \cite[Theorem 4.1]{MR2373102}). This implies the well-definedness of $\{Y_t(\psi)\}_t$ and we have the following analogue of \cite[Proposition 3.5]{MR3647575}.

\begin{proposition}
Let $\psi, \{Z_t^z\}_t, \{Y_t(\psi)\}_t$ be as above. Let $p > \frac{d}{4}$. Then 
$\{Y_t(\psi)\}_t$ is an $(\F_t)_t$ adapted $\Sc_{-p}(\R^d)$ valued
process with paths which are right continuous with left limits. Furthermore, $Y_t(\psi)$ is norm-bounded, where the bound can be 
chosen to be independent of $t$.
\end{proposition}

Consider 
the linear map $Z_t(\omega): \Sc(\R^d) \to \Sc(\R^d)$ defined by 
$(Z_t(\omega)\phi)(z):=\phi(Z_t^z(\omega)),\, z \in \R^d$. Now, we have the following analogue of \cite[Lemma 3.6]{MR3647575}.
\begin{lemma}\label{composition-defined}
Fix any $t\geq 0, \omega\in \Omega\setminus\mathcal{N}$. The linear map 
$Z_t(\omega): \Sc(\R^d) \to \Sc(\R^d)$ is continuous.
\end{lemma}

Observe that for any $\phi \in \Sc(\R^d),\psi \in L^1(\R^d)$, we have
\[\inpr{Y_t(\psi)}{\phi} = \int_{\R^d}\psi(z)\,\phi(Z_t^z)\, 
dz = \int_{\R^d}\psi(z)\,(Z_t(\phi))(z)\, dz = \inpr{\psi}{Z_t(\phi)}.\]
This implies
\begin{equation}\label{YasZ*}
Y_t(\psi) = Z^{\ast}_t(\psi),
\end{equation}
where $Z^{\ast}_t(\omega): \Sc^\prime(\R^d)\to \Sc^\prime(\R^d)$ is the transpose of the linear map $Z_t(\omega)$. The next result is the main result of the second half of this paper and is an analogue of \cite[Theorem 3.8]{MR3647575}.

\begin{theorem}\label{SDE-Y}
Let $p > \frac{d}{4}$ and $\psi \in L^1(\R^d)$. Then the 
$\Sc_{-p}(\R^d)$ valued process $\{Y_t(\psi)\}_t$ satisfies the SPDE \eqref{general-SPDE}. This solution is also unique.
\end{theorem}

The proof of the above theorem has been discussed in Section 3. The existence part follows from It\^{o} formula \cite[Theorem 4.4.7]{MR2512800} and duality arguments. The proof of uniqueness uses the Monotonicity inequality (see \cite[Theorem 2.1]{MR2590157} and \cite[Theorem 3.1  and Theorem 4.2]{MR3331916}). Theorem \ref{SDE-Y} for $F \equiv 0$ is exactly \cite[Theorem 3.8]{MR3647575}. Another example of an SPDE in the class considered here is as follows
\begin{align*}
{Y}_t &= \psi + \int_0^t \partial{Y}_{s-}dB_s + \int_0^t \frac{1}{2}\partial^2{Y}_{s-}ds\\
&+ \int_{0<|x|<1}(\tau_{x}-Id){Y}_{t-} \widetilde{N}(dt,dx) +\int_{0<|x|<1}(\tau_{x} - Id +  x\partial){Y}_{t-}\nu(dx)dt, t \geq 0,
\end{align*}
where $d = 1, \sigma \equiv 1, b \equiv 0, F(0, x) = x, \forall x$.

An important observation is that $Y_t(\psi)$, as above, is given by a (random) $L^1(\R^d)$ function. This is an analogue of \cite[Proposition 3.11]{MR3647575} and has been stated in Proposition \ref{Y-as-integrable-function}.

Using an analogue of the interlacing technique \cite[Example 1.3.13 and Theorem 6.2.9]{MR2512800}, we also show the existence and uniqueness of strong solutions for SPDEs involving large jump terms. 

\begin{theorem}\label{SPDE-with-large-jump}
Let $p > \frac{d}{4}$ and $\psi \in L^1(\R^d)$. Let $G:\{x \in \R^d: |x| \geq 1\} \to \R^d$ be a measurable function. Let $\sigma, b$ and $F$ be as in Theorem \ref{SDE-Y}. Then there exists an $\Sc_{-p}(\R^d)$ valued process $\{Y_t(\psi)\}_t$ satisfying the following SPDE
\begin{equation}\label{SPDE-with-large-jump-equation}\begin{split}
  d{Y}_t = A^*({Y}_{t-})dB_t + L^*({Y}_{t-})dt
+&\int_{0<|x|<1}(\tau_{F(0,x)}-Id){Y}_{t-}\widetilde{N}(dt,dx)+\int_{|x|\geq1}(\tau_{G(x)}-Id){Y}_{t-}N(dt,dx) 
\\
+&\int_{0<|x|<1}(\tau_{F(0,x)}-Id+F(0,x)\partial){Y}_{t-}\nu(dx)dt,
\end{split}
 \end{equation}
with equality in $\Sc_{-p-1}$ and initial condition $\psi$. This solution is also unique.
\end{theorem}

The proof of Theorem \ref{SPDE-with-large-jump} has been discussed in Section 3.

\section{Proof of Theorem \ref{ddi-characterization}}

In this section, we give a sketch of the proof of Theorem \ref{ddi-characterization}. As mentioned in Section 1 Introduction, the proof is provided for the one-dimensional case only.

\begin{proof}
If the relation \eqref{dependence deterministically on initial condition dd} holds, then $f(0,z)=z, \forall z$. Rewriting \eqref{dependence deterministically on initial condition dd}, we have
\begin{align*}
\int_0^t\left(b(Z^z_{s-})-b(Z^0_{s-})-\frac{\partial f}{\partial s}(s,z)\right)ds &+\int^t_0\left(\sigma(Z^z_{s-})-\sigma(Z^0_{s-})\right)dB_s\\
&+\int^t_0\int_{0<|x|<1}\left(F(Z^z_{s-},x)-F(Z^0_{s-},x)\right)\widetilde{N}(ds,dx)=0
\end{align*}
We denote the first, second and the third term on the left hand side by $A_t, M^1_t$ and $M^2_t$ respectively. Note that $\{M^1_t + M^2_t\}_t$ is a martingale with paths which are right continuous and have left limits. But, $\{A_t\}_t$ is a predictable process of integrable variation. Since $M^1_0 + M^2_0 = 0$, using \cite[Lemma 25.11]{MR1876169} we conclude a.s. for all $t \geq 0$, $A_t = 0 = M^1_t + M^2_t$.

Observe that a.s. $[M^1, M^1 + M^2]_t = 0, \forall t \geq 0$.
Since $B$ and $N$ are independent, we have a.s. $[M^1, M^2]_t = 0, \forall t \geq 0$. Hence, a.s. $[M^1]_t = [M^1, M^1]_t = 0, \forall t \geq 0$. Similarly, a.s. $[M^2]_t = 0, \forall t \geq 0$. Using It\^{o} isometry, the identification of $\sigma$ goes as in \cite[Theorem 2.5]{MR3647575}. The identification of $F$ is also similar.

To make the identification of $b$, we start with a.s. $A_t = 0, \forall t \geq 0$. Since the process $\{Z^z_t\}_t$ has only countably many jumps in any time interval $[0, t]$, we have a.s. 
\[A_t = \int_0^t\left(b(Z^z_{s})-b(Z^0_{s})-\frac{\partial f}{\partial s}(s,z)\right)ds= 0, \forall t \geq 0.\]
Now, observe that
\begin{equation}\label{ito formula for b}
    b(Z^z_{t})-b(Z^0_{t})-\frac{\partial f}{\partial t}(t,z)=0, \forall z
\end{equation}
In particular for $t=0$, we have
$b(z) = b(0) - \frac{\partial f}{\partial t}(0,z), \forall z$. Since $b$ is a $C^2$ function on $\R$, we expand $b(Z^z_t)$ by It\^{o} formula \cite[Theorem 4.4.7]{MR2512800} and look at the martingale component of $b(Z^z_{t})-b(Z^0_{t})$. Applying \cite[Lemma 25.11]{MR1876169}, we conclude that $b(z) = Cz + k, \forall z$, where $C = b^\prime(0)$ and $k$ is some real number. This part of the argument is similar to the proof of \cite[Theorem 2.5]{MR3647575}.

Now, we have $f(0,z)= z, \forall z$ and
\[f(t,z)- z=\int_0^t\left(b(Z^z_{s})-b(Z^0_{s})\right)ds = \int_0^t C\left(Z^z_{s} - Z^0_{s}\right)ds=\int_0^t Cf(s,z)ds\]
and hence, $f(t,z)=ze^{tC}, \forall t \geq 0, z \in \R$. This completes the proof of necessity.

To prove the converse, observe that there exists a $\pp$ null set $\mathcal{N}$ such that for all $\omega\in \Omega\setminus\mathcal{N},t\geq 0$
\[Z^0_t=  \int^t_0\sigma dB_s + \int^t_0(CZ^0_{s-}+K) ds +\int^t_0\int_{ 0<|x| < 1 } F(0 , x ) \widetilde{ N }( ds,dx)\]
Hence on $\Omega \smallsetminus\mathcal{N}$, $\{Z^0_t+ze^{tC}\}_t$ solves the SDE \eqref{general-SDE}. This completes the proof.
\end{proof}

\begin{remark}
In the SDE related part of this paper, we have not considered the large jumps case. Usually the existence and uniqueness of strong solutions of SDEs of the form
\begin{equation}\label{SDE-with-large-jump}
\begin{split}
dZ_t &= b( Z_{t-} ) dt + \sigma( Z_{t-} ) dB_t + \int_{ 0<|x| < 1 } F(Z_{ t - } , x ) \widetilde{ N }( dt,dx) + \int_{ |x| \geq 1 } G(Z_{ t - } , x ) N( dt,dx),\\
Z_0 &= z,
\end{split}
\end{equation}
are obtained by interlacing arguments  \cite[Example 1.3.13 and Theorem 6.2.9]{MR2512800}. The deterministic dependence on initial conditions in this cases can be worked out in simple situations. Let us work in one dimension and in the pure jump case with $\sigma(\cdot) \equiv b(\cdot) \equiv F(\cdot, \cdot) \equiv 0$. Assume that $z \mapsto G(z, x)$ is continuous for all $x$ with $|x| \geq 1$. Here, as a necessary condition for deterministic dependence on the initial condition, we have $f(t, z) = Z^z_t - Z^0_t = z, \forall t < \tau_1$, where $\{\tau_n\}_n$ are the arrival times of the compound Poisson process given by $A(t) = \int_{|x| \geq 1} x\, N(t, dx), \forall t$. At $t = \tau_1$, by \cite[Theorem 6.2.9]{MR2512800}, we have for every fixed $z \in \R$
\[f(\tau_1, z) = z + [G(z, \bigtriangleup P(\tau_1)) - G(0, \bigtriangleup P(\tau_1))].\]
Since, $\bigtriangleup P(\tau_1)$ is independent of the Poisson process $N(t, \{x: |x| \geq 1\})$ (see the proof of \cite[Theorem 2.3.9]{MR2512800}), we conclude that $f(\tau_1, z)$ is dependent only on $z$. Since $\tau_1$ is an exponential random variable with mean $\frac{1}{\nu(x: |x| \geq 1)}$, we conclude $f(t, z) = z, \forall t > 0, z \in \R$. This implies, $G(z, \bigtriangleup P(\tau_1)) = G(0, \bigtriangleup P(\tau_1)), \forall z$ and hence $(z, x) \mapsto G(z, x)$ depends only on $x$, $\nu$-almost every $z$. Conversely, if $(z, x) \mapsto G(z, x)$ depends only on $x$, then the relation $f(t, z) = Z^z_t - Z^0_t$ is satisfied with $f(t, z) = z, \forall t, z$.
\end{remark}

\section{Proofs of Theorem \ref{SDE-Y} and Theorem \ref{SPDE-with-large-jump}}
This section is devoted to the proofs of Theorems \ref{SDE-Y} and \ref{SPDE-with-large-jump}.

\begin{proof}[Proof of Theorem \ref{SDE-Y}]
Structurally, the proof in higher dimensions include appropriate notational modifications from the proof in dimension one. Since the argument remains the same, to avoid notational complexity, the proof is discussed in dimension one only. 

Fix $\phi \in \Sc(\R)$ and $z \in \R$. Then, by It\^{o} formula \cite[Theorem 4.4.7]{MR2512800}, we have a.s.
\begin{align*}
  Z_t(\phi)(z)&=\phi(z)+ \int^t_0(Z_{-s}(L\phi))(z)ds+\int^t_0(Z_{s-}(A\phi))(z)dB_s\\
  &+\int^t_0\int_{0<|x|<1}\left\{Z_{s-}(\tau_{-F(0,x)}\phi)-Z_{s-}(\phi)\right\}(z)\widetilde{N}(ds,dx)\\
 &+\int^t_0\int_{0<|x|<1}\left\{Z_{s-}(\tau_{-F(0,x)}\phi)-Z_{s-}(\phi)-Z_{s-}(F(0,x)\partial\phi)\right\}(z)\nu(dx)ds.
\end{align*}
Hence, using \eqref{YasZ*}, we have, a.s.
\begin{align*}
\inpr{Y_t(\psi)}{\phi} &= \inpr{Z_t^*(\psi)}{\phi}\\
&= \inpr{\psi}{\phi}+\int^t_0\inpr{\psi}{Z_{s-}(A\phi)}dB_s+\int^t_0\int_{0<|x|<1}\left\langle\psi,Z_{s-}(\tau_{-F(0,x)}\phi)-Z_{s-}(\phi)\right\rangle\widetilde{N}(ds,dx)\\
+&\int^t_0\left\langle\psi,Z_{s-}(L\phi)\right\rangle ds+\int^t_0\int_{0<|x|<1}\left\langle\psi,Z_{s-}(\tau_{-F(0,x)}\phi)-Z_{s-}(\phi)-Z_{s-}(F(0,x)\partial\phi)\right\rangle\nu(dx)ds\\
&=\langle\psi,\phi\rangle+\left\langle \int^t_0A^*(Z_{s-}^*(\psi)dB_s,\phi\right\rangle +\left\langle \int^t_0L^*(Z_{s-}^*(\psi)ds,\phi\right\rangle \\
+&\left\langle\int^t_0\int_{0<|x|<1}(\tau_{F(0,x)}-Id)Z_{s-}^*(\psi) \widetilde{N}(ds,dx),\phi\right\rangle\\
&+\left\langle\int^t_0\int_{0<|x|<1}(\tau_{F(0,x)}-Id+F(0,x)\partial)Z_{s-}^*(\psi)\nu(dx)ds,\phi\right\rangle.
\end{align*}
Since the above result is true for all $\phi \in \Sc(\R)$, we conclude that $\{Y_t(\psi)\}_t$ solves \eqref{general-SPDE} with equality in $\Sc_{-p-1}(\R)$.

Now, we prove the uniqueness of strong solutions. Let $\{Y^1_t\}_t$ and $\{Y^2_t\}_t$ be two $\Sc_{-p}(\R)$ valued solutions of the SPDE \eqref{general-SPDE} with equality in $\Sc_{-p-1}(\R)$ and set $V_t = Y^1_t-Y^2_t, \forall t \geq 0$. Then $\{V_t\}_t$ is $\Sc_{-p}(\R)$ valued and solves the SPDE \eqref{general-SPDE} with initial condition $0$, i.e. a.s.
\begin{align*}
   V_t = \int^t_0L^*V_{s-}ds&+\int_0^tA^*V_{s-}dB_s+\int^t_0\int_{0<|x|<1}\left(\tau_{F(0,x)}-Id \right)V_{s-}\widetilde{N}(ds,dx)\\
&+\int^t_0\int_{0<|x|<1}\left(\tau_{F(0,x)}-Id+F(0,x)\partial \right)V_{s-}\nu(dx)ds 
\end{align*}

By the It\^{o} formula for $\|\cdot\|_{-p-1}^2$, we have a.s.
\begin{align*}
&\left\| V_t\right\|_{-p-1}^2\\
=&\int^t_0\left\{2\left\langle V_{s-}, L^*V_{s-} \right\rangle_{-p-1}+\left\| A^*V_{s-}\right\|^2_{-p-1} \right\} ds+\int_0^t2\left\langle V_{s-}, A^*V_{s-} \right\rangle_{-p-1} dB_s\\
&+\int^t_0\int_{0<|x|<1}\left\{\left\|\tau_{F(0,x)}V_{s-} \right\|^2_{-p-1}-\left\| V_{s-} \right\|^2_{-p-1} \right\}\widetilde{N}(ds,dx)\\
&+\int^t_0\int_{0<|x|<1}\left\{\left\| \tau_{F(0,x)}V_{s-} \right\|^2_{-p-1}-\left\| V_{s-} \right\|^2_{-p-1}+2\left\langle F(0,x)\partial V_{s-},V_{s-}\right\rangle_{-p-1} \right\}\nu(dx)ds 
 \end{align*}
Then, for any $T > 0$, 
\begin{equation}\label{inequality1}
\begin{split}
&\Exp\left\{\sup_{0\leq t\leq T}\left\| V_t\right\|_{-p-1}^2\right\}\\ &\leq \Exp\left\{\sup_{0\leq t\leq T}\int^t_0\left\{2\left\langle V_{s-}, L^*V_{s-} \right\rangle_{-p-1}+\left\| A^*V_{s-}\right\|^2_{-p-1} \right\} ds\right\}\\
&+\Exp\left\{\sup_{0\leq t\leq T}\left|\int_0^t2\left\langle V_{s-}, A^*V_{s-} \right\rangle_{-p-1} dB_s\right|\right\}\\
&+\Exp\left\{\sup_{0\leq t\leq T}\left|\int^t_0\int_{0<|x|<1}\left\{\left\|\tau_{F(0,x)}V_{s-} \right\|^2_{-p-1}-\left\| V_{s-} \right\|^2_{-p-1} \right\}\widetilde{N}(ds,dx)\right|\right\}\\
&+\Exp \left\{\sup_{0\leq t\leq T}\int^t_0\int_{0<|x|<1}\left\{\left\| \tau_{F(0,x)}V_{s-} \right\|^2_{-p-1}-\left\| V_{s-} \right\|^2_{-p-1}+2\left\langle F(0,x)\partial V_{s-},V_{s-}\right\rangle_{-p-1} \right\}\nu(dx)ds.\right\}
\end{split}
\end{equation}
We denote the four terms on the right hand side above by $I_1, I_2, I_3$ and $I_4$. We now observe the following inequalities. Constants appearing in the computations below depend on one or more of the parameters $p, \sigma, b$ and $F$ only and may change their values from line to line.
\begin{enumerate}[label=(\roman*)]
\item By the Monotonicity inequality for $A^\ast, L^\ast$ \cite[Theorem 4.2]{MR3331916}, we have   \begin{align*}
     I_1= & \Exp\left\{\sup_{0\leq t\leq T}\int^t_0\left\{2\left\langle V_{s-}, L^*V_{s-} \right\rangle_{-p-1}+\left\| A^*V_{s-}\right\|^2_{-p-1} \right\} ds\right\}\\
   \leq & \Exp\left\{\sup_{0\leq t\leq T}\int^t_0 C_1\|V_{s-}\|_{-p-1}^2 ds\right\}\\
   \leq& C_1\int^T_0 \Exp\left\{\sup_{0\leq t\leq s}\|V_{t}\|_{-p-1}^2\right\}ds.
 \end{align*}

\item It was observed in \cite[Proof of Theorem 3.1]{MR3331916} that there exists a bounded linear operator $\mathbb{T}:\Sc_{-p - 1}(\R) \to \Sc_{-p - 1}(\R)$ such that
\begin{equation}\label{first-order-monotonicity}
\inpr[-p-1]{\phi}{\partial \phi} = \inpr[-p-1]{\phi}{\mathbb{T} \phi}, \forall \phi \in \Sc_{-p - \frac{1}{2}}(\R).
\end{equation} 
By the Burkholder-Davis-Gundy inequality \cite[Theorem 26.12]{MR1876169} and \eqref{first-order-monotonicity}, we have
 \begin{align*}
     I_2 =& 2\Exp\left\{\sup_{0\leq t\leq T}\left|\int_0^t\left\langle V_{s-}, A^*V_{s-} \right\rangle_{-p-1} dB_s\right|\right\}\\
\leq & C_2\Exp\left(\int_0^T\left\langle V_{s-}, A^*V_{s-} \right\rangle_{-p-1}^2 ds\right)^\frac{1}{2}\\
= & C_2\Exp\left(\int_0^T\sigma^2\left\langle V_{s-}, \partial V_{s-} \right\rangle_{-p-1}^2 ds\right)^\frac{1}{2}\\
\leq &C_2\Exp\left(\int_0^T \left\|V_{s-}\right\|^4_{-p-1} ds\right)^\frac{1}{2}\\
\leq & C_2\Exp\left(\sup_{0\leq t\leq T}\left\|V_{t}\right\|^2_{-p-1}\int_0^T \left\|V_{s-}\right\|^2_{-p-1} ds\right)^\frac{1}{2}\\
\leq & C_2\left(\epsilon_1\Exp\left(\sup_{0\leq t\leq T}\left\|V_{t}\right\|^2_{-p-1}\right)+\frac{1}{\epsilon_1}\Exp\left(\int_0^T \left\|V_{s-}\right\|^2_{-p-1} ds\right)\right)\\
\leq & C_2\left(\epsilon_1\Exp\left(\sup_{0\leq t\leq T}\left\|V_{t}\right\|^2_{-p-1}\right)+\frac{1}{\epsilon_1}\int_0^T \Exp\left(\sup_{0\leq t\leq s}\left\|V_t\right\|^2_{-p-1} \right)ds\right),
 \end{align*} 
for any $\epsilon_1 > 0$.

\item Let $\{h_n: n \geq 0\}$ be orthonormal basis of $L^2(\R)$ formed by the Hermite functions  
\cite[Chapter 1.3]{MR771478}. Note that $\{(2n+1)^{-p} h_n: n \geq 0\}$ is an orthonormal basis for $\Sc_p(\R)$.

For fixed $z\in \R, \phi \in \Sc_{-p}(\R), m \geq 0$, the following $C^2$ function $g: [0, 1] \to \R$ defined by $g(v;z,\phi,m) := \inpr{\tau_{vz}\phi}{h_m}^2,\forall v\in [0,1]$ was considered in \cite[Lemma 2.5]{MR4148176}. Using \eqref{first-order-monotonicity}, we have
\begin{align*}
\sum_{m = 1}^\infty (2m+1)^{-2(p + 1)} g^\prime(v;z,\phi,m) &= -2\sum_{m = 1}^\infty(2m+1)^{-2(p + 1)} \inpr{\tau_{vz}\phi}{h_m}\inpr{z\partial\tau_{vz}\phi}{h_m}\\
&=-2 \inpr[-p-1]{\tau_{vz}}{z\partial \tau_{vz}}\\
&= -2z \inpr[-p-1]{\tau_{vz}}{\mathbb{T} \tau_{vz}}, \forall v \in (0, 1).
\end{align*}
Using this observation and the boundedness of the translation operators \cite[Theorem 2.1]{MR1999259}, we have
\begin{equation}\label{first-order-taylor-reqd}
\begin{split}
     &\left(\left\|\tau_{F(0,x)}V_{s-} \right\|^2_{-p-1}-\left\| V_{s-} \right\|^2_{-p-1}\right)^2\\ =&\left(\sum_{m=0}^\infty(2m+1)^{-2(p+1)}[f\left(1;F(0,x),V_{s-},m\right)-f\left(0;F(0,x),V_{s-},m\right)]\right)^2\\
 =&\left(\sum_{m=0}^\infty(2m+1)^{-2(p+1)}\int^1_0f'\left(v;F(0,x),V_{s-},m\right)dv\right)^2\\
  \leq & \; 4 (F(0, x))^2 \int^1_0\left\langle\tau_{vF(0,x)}V_{s-},\mathbb{T}\tau_{vF(0,x)}V_{s-}\right\rangle_{-p-1}^2 dv\\
  \leq & \; 4 (F(0, x))^2 \int^1_0\left(P_k(|vF(0,x)|)\left\|V_{s-}\right\|_{-p-1}\right)^4 dv\\
\leq & \; 4 \left(\sup_{0\leq |x|\leq 1}\left(P_k(|F(0,x)|)\right)^4\right) (F(0, x))^2 \left\|V_{s-}\right\|_{-p-1}^4,
 \end{split}
 \end{equation}
for some real polynomial $P_k$ of degree $k = 2([|p + 1|] + 1)$ with non-negative coefficients.

Applying the Burkholder-Davis-Gundy inequality \cite[Theorem 26.12]{MR1876169}, \eqref{first-order-taylor-reqd},  \eqref{integrable-F} and \eqref{bound-F} we have
 \begin{align*}
     I_3=&\Exp\left\{\sup_{0\leq t\leq T}\left|\int^t_0\int_{0<|x|<1}\left\{\left\|\tau_{F(0,x)}V_{s-} \right\|^2_{-p-1}-\left\| V_{s-} \right\|^2_{-p-1} \right\}\widetilde{N}(ds,dx)\right|\right\}\\
\leq &C_3\Exp\left(\int^T_0\int_{0<|x|<1}\left(\left\|\tau_{F(0,x)}V_{s-} \right\|^2_{-p-1}-\left\| V_{s-} \right\|^2_{-p-1} \right)^2\nu(dx)ds\right)^\frac{1}{2}\\
 \leq &C_3\Exp\left(\int^T_0\int_{0<|x|<1}F(0,x)^2\left\|V_{s-}\right\|_{-p-1}^4 \nu(dx)ds\right)^\frac{1}{2}\\
  \leq & C_3\left(\int_{0<|x|<1}F(0,x)^2\nu(dx)\right)^\frac{1}{2}\left(\epsilon_2\Exp\left(\sup_{0\leq t\leq T}\left\|V_{t}\right\|^2_{-p-1}\right)+\frac{1}{\epsilon_2}\int_0^T \Exp\left(\sup_{0\leq t\leq s}\left\|V_t\right\|^2_{-p-1} \right)ds\right)\\
  \leq&C_3\left[\epsilon_2\Exp\left(\sup_{0\leq t\leq T}\left\|V_t\right\|^2_{-p-1}\right)+\frac{1}{\epsilon_2}\int_0^T \Exp\left(\sup_{0\leq t\leq s}\left\|V_t\right\|^2_{-p-1} \right)ds)\right],
 \end{align*}
for any $\epsilon_2 > 0$.

\item The relevant estimate for $I_4$ is obtained similar to the estimate for Term 2 in the proof of \cite[Lemma 3.4]{MR4148176}. The difference with the present setup with that of \cite{MR4148176} is that the function $F$ there depended on both its arguments. In the present case, we use \eqref{integrable-F} and \eqref{bound-F} and obtain
\[I_4 \leq C_4\int_0^T \Exp\left(\sup_{0\leq t\leq s}\left\|V_t\right\|^2_{-p-1} \right)ds.\]
\end{enumerate}

Using these derived inequalities in \eqref{inequality1}, we end up with
\begin{align*}
&\Exp\{\sup_{0\leq t\leq T}\left\| V_t\right\|_{-p-1}^2\}\\
\leq& \left(C_1+C_4+\frac{C_2}{\epsilon_1}+\frac{C_3}{\epsilon_2}+\right)\int^T_0 \Exp\left\{\sup_{0\leq t\leq s}\|V_{t}\|_{-p-1}^2\right\}ds +\left(C_2\epsilon_1+C_3\epsilon_2\right)\Exp\left(\sup_{0\leq t\leq T}\left\|V_{t}\right\|^2_{-p-1}\right)
 \end{align*}
We choose $\epsilon_1 > 0$ and $\epsilon_2 > 0$ such that $\left(1-C_2\epsilon_1-C_3\epsilon_2\right) > 0$ and we have
\[\Exp\{\sup_{0\leq t\leq T}\left\| V_t\right\|_{-p-1}^2\}\leq  C_5 \int^T_0 \Exp\left\{\sup_{0\leq t\leq s}\|V_{t}\|_{-p-1}^2\right\}ds,\]
for some constant $C_5 > 0$ depending only on $p, \sigma, b$ and $F$.

Using Gronwall's inequality we get a.s. for all $t \in [0, T]$, $V_t=Y^1_t-Y^2_t=0$. The result follows.
\end{proof}

An important observation useful for the proof of Theorem \ref{SPDE-with-large-jump} is the following result that $Y_t(\psi)$, as in Theorem \ref{SDE-Y}, is given by a (random) integrable function. This is an analogue of \cite[Proposition 3.11]{MR3647575}.

\begin{proposition}\label{Y-as-integrable-function}
Under the assumptions of Theorem \ref{SDE-Y}, on $\omega \in \Omega\setminus\mathcal{N}$, the distribution $Y_t(\psi)$ is given by the $L^1(\R^d)$ function \[e^{-t\,tr(C)}\,\tau_{Z_t^0}\psi_t(\cdot),\]
where $\psi_t(z):=\psi(e^{-tC}z)$ for $t \geq 0, z 
\in 
\R^d$ and $tr(C)$ is the trace of the matrix $C$. 
\end{proposition}
\begin{proof}
For any $\phi \in \Sc(\R^d)$, we have
\begin{align*}
\inpr{Y_t(\psi)}{\phi} &= 
\int_{\R^d} \psi(z) \phi(e^{tC}z + Z_t^0) \, dz\\
&= |det(e^{-tC})|\int_{\R^d} \psi(e^{-tC}(z-Z_t^0)) \phi(z) \, dz,\\
&= e^{-t\,tr(C)} \int_{\R^d} (\tau_{Z_t^0}\psi_t)(z) \phi(z) \, dz.
\end{align*}
In the last step above, we have used the equality $|det(e^{-tC})| = e^{-t\,tr(C)}$ (see
\cite[Problem 5.6.P43]{MR2978290}). This completes the proof.
\end{proof}

As a consequence of the above proposition, we obtain the next result.

\begin{theorem}\label{SPDE-with-random-initial}
Let $\psi$ be an $L^1(\R^d)$ valued $\F_0$ measurable random variable. Let $\sigma, b$ and $F$ be as in Theorem \ref{SDE-Y}. Fix $p > \frac{d}{4}$. Then, the SPDE \eqref{general-SPDE} with initial condition $\psi$ has a unique $\Sc_{-p}(\R^d)$ strong solution.
\end{theorem}

\begin{proof}
First, assume that there exists $\alpha > 0$ such that $\|\psi(\omega)\|_{L^1{\R^d}} \leq \alpha, \forall \omega$. Define 
\begin{equation}\label{Y-psi-general}
Y_t(\omega)(\psi) := \int_{\R^d} (\psi(\omega))(z) \delta_{Z^z_t(\omega)}\, dz, \forall t \geq 0, \omega \in \Omega\setminus\mathcal{N}
\end{equation}
and set $Y_t(\omega)(\psi) = 0, \forall t \geq 0, \omega \in \mathcal{N}$, where $\mathcal{N}$ is as in \eqref{Y-psi}. As proved in Theorem \ref{SDE-Y}, it can be shown that $\{Y_t(\psi)\}_t$ is an $\Sc_{-p}(\R^d)$ valued adapted process satisfying the SPDE \eqref{general-SPDE} with equality in $\Sc_{-p-1}(\R^d)$.

Now, let $\psi$ be any $L^1(\R^d)$ valued $\F_0$ measurable random variable. For positive integers $k$, consider $\psi^{(k)}:\Omega \to L^1(\R^d)$ defined by $\psi^{(k)} = \psi$ on the set $\{\|\psi\|_{L^1(\R^d)} \leq k\}$ and $0$ otherwise. Consider the processes $\{Y_t(\psi^{(k)})\}_t, k = 1, 2, \cdots$ which solve the SPDE \eqref{general-SPDE} with initial condition $\psi^{(k)}$. Observe that the consistency property $Y_t(\psi^{(k)}) = Y_t(\psi^{(n)}), \forall n \geq k, \forall t \geq 0$ on the set $\{\|\psi\|_{L^1(\R^d)} \leq k\}$ holds and $\{\|\psi\|_{L^1(\R^d)} \leq k\} \uparrow \Omega$ as $k \to \infty$. Hence, $\{Y_t(\psi)\}_t$, defined by \eqref{Y-psi-general} for general $\psi$ solves the SPDE \eqref{general-SPDE} with initial condition $\psi$. Uniqueness of the solution follows from the consistency criteria discussed above.
\end{proof}

\begin{remark}\label{Y-as-integrable-function-2}
As discussed in Proposition \ref{Y-as-integrable-function} for deterministic $\psi$, we can show that $Y_t(\psi)$ obtained in Theorem \ref{SPDE-with-random-initial} for random $\psi$ is a (random) $L^1(\R^d)$ function.
\end{remark}

\begin{proof}[Proof of Theorem \ref{SPDE-with-large-jump}]
We use an analogue of interlacing argument in \cite[Theorem 6.2.9]{MR2512800} to construct a solution. 

Let $\{P_t\}_t$ be the compound Poisson process defined by $P(t)=\int_{|x|\geq1}xN(t,dx), \forall t$. Let $\tau_n, n = 1, 2, \cdots$ denote the arrival times of the jumps of $\{P_t\}_t$. Let $\{Y_t(\psi)\}_t$ denote the $\Sc_{-p}(\R^d)$ valued strong solution of SPDE \eqref{general-SPDE} with initial condition $\psi$, where $\psi$ is an $L^1(\R^d)$ valued random variable.

Consider the process $\{\tilde Y_t(\psi)\}_t$ defined as follows.
\begin{equation}\label{solution-by-interlacing}
\begin{split}
&\tilde{Y}_t(\psi) := Y_t(\psi), \text{ if } 0\leq t<\tau_1\\
&\tilde{Y}_{\tau_1}(\psi) := \tilde{Y}_{\tau_1-}(\psi)+\left(\tau_{G(\bigtriangleup P_{\tau_1})}-Id\right)\tilde{Y}_{\tau_1-}(\psi)\\
&\tilde{Y}_t(\psi)=\tilde{Y}_{\tau_1}(\psi)+Y_t(\tilde{Y}_{\tau_1}(\psi))-Y_{\tau_1}(\tilde{Y}_{\tau_1}(\psi)), \text{ if } \tau_1 < t < \tau_2,
\end{split}
\end{equation}
where $\bigtriangleup P_{\tau_1}$ gives the first jump of $\{P_t\}_t$. Note that $\tilde{Y}_{\tau_1}(\psi)$ is given by a random $L^1(\R^d)$ function (see Remark \ref{Y-as-integrable-function-2}) and $\{Y_t(\tilde{Y}_{\tau_1})\}_t$ denotes the unique solution to \eqref{general-SPDE} with initial condition $\tilde{Y}_{\tau_1}$ (see Theorem \ref{SPDE-with-random-initial}).

Continuing as in \eqref{solution-by-interlacing}, recursively we define $\tilde Y_t(\psi)$ for all $t \geq 0$. Note that the process $\{\tilde Y_t(\psi)\}_t$ is adapted, $\Sc_{-p}(\R^d)$ valued and has paths which are right continuous with left limits. Then, $\{\tilde Y_t(\psi)\}_t$ solves \eqref{SPDE-with-large-jump-equation}. Uniqueness of the solution follows from standard arguments.
\end{proof}

\noindent\textbf{Acknowledgements:} The first author would like to acknowledge the fact that he was supported by the University Grants Commission (Government of India) Ph.D research Fellowship. The second author would like to acknowledge the fact that he was partially supported by the Matrics grant MTR/2021/000517 from the Science and Engineering Research Board (Department of Science \& Technology, Government of India).

\bibliographystyle{amsplain}
\bibliography{ref}

\providecommand{\bysame}{\leavevmode\hbox to3em{\hrulefill}\thinspace}
\providecommand{\MR}{\relax\ifhmode\unskip\space\fi MR }
\providecommand{\MRhref}[2]{%
  \href{http://www.ams.org/mathscinet-getitem?mr=#1}{#2}
}
\providecommand{\href}[2]{#2}
\begin{thebibliography}{10}

\bibitem{MR1276439}
David Applebaum, \emph{Unitary actions of {L}\'{e}vy flows of diffeomorphisms},
  J. Multivariate Anal. \textbf{49} (1994), no.~2, 266--277. \MR{1276439}

\bibitem{MR2512800}
\bysame, \emph{L\'{e}vy processes and stochastic calculus}, second ed.,
  Cambridge Studies in Advanced Mathematics, vol. 116, Cambridge University
  Press, Cambridge, 2009. \MR{2512800}

\bibitem{MR1251218}
David Applebaum and Hiroshi Kunita, \emph{L\'{e}vy flows on manifolds and
  {L}\'{e}vy processes on {L}ie groups}, J. Math. Kyoto Univ. \textbf{33}
  (1993), no.~4, 1103--1123. \MR{1251218}

\bibitem{MR866340}
Peter Baxendale and Theodore~E. Harris, \emph{Isotropic stochastic flows}, Ann.
  Probab. \textbf{14} (1986), no.~4, 1155--1179. \MR{866340}

\bibitem{MR3647575}
Suprio Bhar, \emph{Characterizing {G}aussian flows arising from {I}t\^{o}'s
  stochastic differential equations}, Potential Anal. \textbf{46} (2017),
  no.~2, 261--277. \MR{3647575}

\bibitem{MR4148176}
Suprio Bhar, Rajeev Bhaskaran, and Barun Sarkar, \emph{Stochastic {PDE}s in
  {$\mathcal{S}'$} for {SDE}s driven by {L}\'{e}vy noise}, Random Oper. Stoch.
  Equ. \textbf{28} (2020), no.~3, 217--226. \MR{4148176}

\bibitem{MR3331916}
Suprio Bhar and B.~Rajeev, \emph{Differential operators on {H}ermite {S}obolev
  spaces}, Proc. Indian Acad. Sci. Math. Sci. \textbf{125} (2015), no.~1,
  113--125. \MR{3331916}

\bibitem{MR608026}
Jean-Michel Bismut, \emph{A generalized formula of {I}t\^{o} and some other
  properties of stochastic flows}, Z. Wahrsch. Verw. Gebiete \textbf{55}
  (1981), no.~3, 331--350. \MR{608026}

\bibitem{MR2674056}
Vladimir Bogachev, Giuseppe Da~Prato, and Michael R\"{o}ckner, \emph{Existence
  and uniqueness of solutions for {F}okker-{P}lanck equations on {H}ilbert
  spaces}, J. Evol. Equ. \textbf{10} (2010), no.~3, 487--509. \MR{2674056}

\bibitem{MR2857016}
\bysame, \emph{Existence results for {F}okker-{P}lanck equations in {H}ilbert
  spaces}, Seminar on {S}tochastic {A}nalysis, {R}andom {F}ields and
  {A}pplications {VI}, Progr. Probab., vol.~63, Birkh\"{a}user/Springer Basel
  AG, Basel, 2011, pp.~23--35. \MR{2857016}

\bibitem{MR2765423}
Vladimir Bogachev, Giuseppe~Da Prato, and Michael R\"{o}ckner, \emph{Uniqueness
  for solutions of {F}okker-{P}lanck equations on infinite dimensional spaces},
  Comm. Partial Differential Equations \textbf{36} (2011), no.~6, 925--939.
  \MR{2765423}

\bibitem{MR805125}
Andrew Carverhill, \emph{Flows of stochastic dynamical systems: ergodic
  theory}, Stochastics \textbf{14} (1985), no.~4, 273--317. \MR{805125}

\bibitem{MR3236753}
Giuseppe Da~Prato and Jerzy Zabczyk, \emph{Stochastic equations in infinite
  dimensions}, second ed., Encyclopedia of Mathematics and its Applications,
  vol. 152, Cambridge University Press, Cambridge, 2014. \MR{3236753}

\bibitem{MR2460554}
Giulia Di~Nunno, Bernt \O~ksendal, and Frank Proske, \emph{Malliavin calculus
  for {L}\'{e}vy processes with applications to finance}, Universitext,
  Springer-Verlag, Berlin, 2009. \MR{2460554}

\bibitem{MR517235}
K.~D. Elworthy, \emph{Stochastic dynamical systems and their flows}, Stochastic
  analysis ({P}roc. {I}nternat. {C}onf., {N}orthwestern {U}niv., {E}vanston,
  {I}ll., 1978), Academic Press, New York-London, 1978, pp.~79--95. \MR{517235}

\bibitem{MR1093330}
Tsukasa Fujiwara, \emph{Stochastic differential equations of jump type on
  manifolds and {L}\'{e}vy flows}, J. Math. Kyoto Univ. \textbf{31} (1991),
  no.~1, 99--119. \MR{1093330}

\bibitem{MR2479730}
L.~Gawarecki, V.~Mandrekar, and B.~Rajeev, \emph{Linear stochastic differential
  equations in the dual of a multi-{H}ilbertian space}, Theory Stoch. Process.
  \textbf{14} (2008), no.~2, 28--34. \MR{2479730}

\bibitem{MR2590157}
\bysame, \emph{The monotonicity inequality for linear stochastic partial
  differential equations}, Infin. Dimens. Anal. Quantum Probab. Relat. Top.
  \textbf{12} (2009), no.~4, 575--591. \MR{2590157}

\bibitem{MR2560625}
Leszek Gawarecki and Vidyadhar Mandrekar, \emph{Stochastic differential
  equations in infinite dimensions with applications to stochastic partial
  differential equations}, Probability and its Applications (New York),
  Springer, Heidelberg, 2011. \MR{2560625}

\bibitem{MR2978290}
Roger~A. Horn and Charles~R. Johnson, \emph{Matrix analysis}, second ed.,
  Cambridge University Press, Cambridge, 2013. \MR{2978290}

\bibitem{MR776981}
Nobuyuki Ikeda and Shinzo Watanabe, \emph{Stochastic flows of diffeomorphisms},
  Stochastic analysis and applications, Adv. Probab. Related Topics, vol.~7,
  Dekker, New York, 1984, pp.~179--198. \MR{776981}

\bibitem{MR771478}
Kiyosi It\"{o}, \emph{Foundations of stochastic differential equations in
  infinite-dimensional spaces}, CBMS-NSF Regional Conference Series in Applied
  Mathematics, vol.~47, Society for Industrial and Applied Mathematics (SIAM),
  Philadelphia, PA, 1984. \MR{771478}

\bibitem{MR1876169}
Olav Kallenberg, \emph{Foundations of modern probability}, second ed.,
  Probability and its Applications (New York), Springer-Verlag, New York, 2002.
  \MR{1876169}

\bibitem{MR1465436}
Gopinath Kallianpur and Jie Xiong, \emph{Stochastic differential equations in
  infinite-dimensional spaces}, Institute of Mathematical Statistics Lecture
  Notes---Monograph Series, vol.~26, Institute of Mathematical Statistics,
  Hayward, CA, 1995, Expanded version of the lectures delivered as part of the
  1993 Barrett Lectures at the University of Tennessee, Knoxville, TN, March
  25--27, 1993, With a foreword by Balram S. Rajput and Jan Rosinski.
  \MR{1465436}

\bibitem{MR570795}
N.~V. Krylov and B.~L. Rozovski\u{\i}, \emph{Stochastic evolution equations},
  Current problems in mathematics, {V}ol. 14 ({R}ussian), Akad. Nauk SSSR,
  Vsesoyuz. Inst. Nauchn. i Tekhn. Informatsii, Moscow, 1979, pp.~71--147, 256.
  \MR{570795}

\bibitem{MR876080}
H.~Kunita, \emph{Stochastic differential equations and stochastic flows of
  diffeomorphisms}, \'{E}cole d'\'{e}t\'{e} de probabilit\'{e}s de
  {S}aint-{F}lour, {XII}---1982, Lecture Notes in Math., vol. 1097, Springer,
  Berlin, 1984, pp.~143--303. \MR{876080}

\bibitem{MR1472487}
Hiroshi Kunita, \emph{Stochastic flows and stochastic differential equations},
  Cambridge Studies in Advanced Mathematics, vol.~24, Cambridge University
  Press, Cambridge, 1997, Reprint of the 1990 original. \MR{1472487}

\bibitem{MR2090755}
\bysame, \emph{Stochastic differential equations based on {L}\'{e}vy processes
  and stochastic flows of diffeomorphisms}, Real and stochastic analysis,
  Trends Math., Birkh\"{a}user Boston, Boston, MA, 2004, pp.~305--373.
  \MR{2090755}

\bibitem{MR3929750}
\bysame, \emph{Stochastic flows and jump-diffusions}, Probability Theory and
  Stochastic Modelling, vol.~92, Springer, Singapore, 2019. \MR{3929750}

\bibitem{MR540035}
Paul Malliavin, \emph{G\'{e}om\'{e}trie diff\'{e}rentielle stochastique},
  S\'{e}minaire de Math\'{e}matiques Sup\'{e}rieures [Seminar on Higher
  Mathematics], vol.~64, Presses de l'Universit\'{e} de Montr\'{e}al, Montreal,
  Que., 1978, Notes prepared by Dani\`ele Dehen and Dominique Michel.
  \MR{540035}

\bibitem{MR622556}
P.-A. Meyer, \emph{Flot d'une \'{e}quation diff\'{e}rentielle stochastique
  (d'apr\`es {M}alliavin, {B}ismut, {K}unita)}, Seminar on {P}robability, {XV}
  ({U}niv. {S}trasbourg, {S}trasbourg, 1979/1980) ({F}rench), Lecture Notes in
  Math., vol. 850, Springer, Berlin-New York, 1981, pp.~103--117. \MR{622556}

\bibitem{MR2356959}
S.~Peszat and J.~Zabczyk, \emph{Stochastic partial differential equations with
  {L}\'{e}vy noise}, Encyclopedia of Mathematics and its Applications, vol.
  113, Cambridge University Press, Cambridge, 2007, An evolution equation
  approach. \MR{2356959}

\bibitem{MR2020294}
Philip~E. Protter, \emph{Stochastic integration and differential equations},
  second ed., Applications of Mathematics (New York), vol.~21, Springer-Verlag,
  Berlin, 2004, Stochastic Modelling and Applied Probability. \MR{2020294}

\bibitem{MR1999259}
B.~Rajeev and S.~Thangavelu, \emph{Probabilistic representations of solutions
  to the heat equation}, Proc. Indian Acad. Sci. Math. Sci. \textbf{113}
  (2003), no.~3, 321--332. \MR{1999259}

\bibitem{MR2373102}
\bysame, \emph{Probabilistic representations of solutions of the forward
  equations}, Potential Anal. \textbf{28} (2008), no.~2, 139--162. \MR{2373102}

\bibitem{MR1135324}
B.~L. Rozovski\u{\i}, \emph{Stochastic evolution systems}, Mathematics and its
  Applications (Soviet Series), vol.~35, Kluwer Academic Publishers Group,
  Dordrecht, 1990, Linear theory and applications to nonlinear filtering,
  Translated from the Russian by A. Yarkho. \MR{1135324}

\bibitem{MR1045162}
P.~Sundar, \emph{Invertibility of stochastic flows and time reversal of
  solutions of {SDE}s driven by {L}\'{e}vy processes with bounded jumps},
  Soochow J. Math. \textbf{15} (1989), no.~2, 193--204. \MR{1045162}

\bibitem{MR2531088}
H.~van Bargen and G.~Dimitroff, \emph{Isotropic {O}rnstein-{U}hlenbeck flows},
  Stochastic Process. Appl. \textbf{119} (2009), no.~7, 2166--2197.
  \MR{2531088}

\end{thebibliography}
 
\end{document}